\documentclass[12pt, reqno]{amsart}
\makeatletter
\@namedef{subjclassname@1991}{$\mathrm{1991}$ Mathematics Subject Classification}
\@namedef{subjclassname@2000}{$\mathrm{2000}$ Mathematics Subject Classification}
\@namedef{subjclassname@2010}{$\mathrm{2010}$ Mathematics Subject Classification}
\@namedef{subjclassname@2020}{$\mathrm{2020}$ Mathematics Subject Classification}
\makeatother
\usepackage{amsmath,amsthm, amscd, amsfonts, amssymb, graphicx, color}
\usepackage[bookmarksnumbered, colorlinks, plainpages,linkcolor=blue,urlcolor=blue,citecolor=blue]{hyperref}
\newtheorem{thm}{Theorem}[section]
\newtheorem{cor}[thm]{Corollary}
\newtheorem{lem}[thm]{Lemma}

\numberwithin{equation}{section}

\begin{document}

\title{The core inverse in a Banach algebra with involution}

\author{Marjan Sheibani}
\author{Huanyin Chen}
\address{Farzanegan Campus, Semnan University, Semnan, Iran}
\email{<m.sheibani@semnan.ac.ir>}
\address{
School of Mathematics\\ Hangzhou Normal University\\ Hang -zhou, China}
\email{<huanyinchenhz@163.com>}


\begin{abstract}  We present new additive results for the core inverse in a Banach algebra with involution. We obtain necessary and sufficient conditions under which
the sum of two core invertible elements in a Banach algebra with involution is core invertible. Then we apply our results to block complex matrices and
obtain certain conditions under which a block complex is core invertible. These generalize many known results, e.g.,~\cite[Theorem 4.3]{XCZ}, ~\cite[Theorem 2.5]{XS}.
\end{abstract}

\maketitle

\section{Introduction}

An involution of a Banach algebra $\mathcal{A}$ is an
anti-automorphism whose square is the identity map $1$. Thus an
involution of a Banach algebra $\mathcal{A}$ is an operation $* :\mathcal{A}\rightarrow \mathcal{A}$ such
that $(x+y)^* = x^*+y^*$, $(xy)^* = y^*x^*$ and $(x^*)^* = x$ for
all $x, y \in \mathcal{A}$. A Banach algebra $\mathcal{A}$ with involution $*$ is called a
Banach *-algebra, and every complex Banach is a  $C^*$-algebra. Let $\mathcal{A}$ be a Banach *-algebra. We say that $a\in \mathcal{A}$ has core inverse if there exists some $x\in \mathcal{A}$ such that $$xa^2=a, ax^2=x, (ax)^*=ax.$$
If such $x$ exists, it is unique, and denote it by $a^{\tiny\textcircled{\#}}$.

An element $a$ in a Banach algebra $\mathcal{A}$ has group inverse provided that there exists $x\in \mathcal{A}$ such that $$a=axa, x=xax, ax=xa.$$ Such $x$ is unique if exists, denoted by $a^{\#}$, and called the group inverse of $a$. As is well known, a square complex matrix $A$ has group inverse if and only if $rank(A)=rank(A^2)$. Group invertibility was also extensively study under the "strongly regularity" in ring theory.

In ~\cite[Theorem 2.6]{XCZ}, it was proved that $a\in \mathcal{A}^{\tiny\textcircled{\#}}$ if and only if $a\in \mathcal{A}^{\#}$ and $a$ has $(1,3)$-inverse. Here $a\in \mathcal{A}$ has $(1,3)$ inverse if and only if there exists some $x\in \mathcal{A}$ such that $a=axa$ and $(ax)^*=ax$. We list several characterizations of core inverse.

\begin{thm} (see~\cite[Theorem 2.14]{RD}, \cite[Theorem 2.8]{RD} and \cite[Theorem 3.4]{LC}. Let $\mathcal{A}$ be a Banach *-algebra, and let $a\in \mathcal{A}$. Then the following are equivalent:\end{thm}
\begin{enumerate}
\item [(1)]{\it $a$ has core inverse.}
\item [(2)]{\it There exists $x\in \mathcal{A}$ such that $axa=a, x=xax, xa^2=a, ax^2=x, (ax)^*=ax$.}
\item [(3)]{\it There exists $x\in \mathcal{A}$ such that $axa=a$ and $aR=xR=x^*R$.}
\item [(4)]{\it There exists a projection $p\in \mathcal{A}$ such that $pa=0$ and $a+p\in \mathcal{A}^{-1}$.}
\item [(5)]{\it $a\in \mathcal{A}^{\#}$ and $\mathcal{A}a=\mathcal{A}a^*a$.}
\end{enumerate}

The core invertibility in a Banach *-algebra is attractive. Many authors have studied such problems from many different views, e.g., ~\cite{BT, RD, K, XS, XCZ, Z2}.
The core inverse of $a+b$ was obtained when $ab=0$ and $a^*b=0$ for two core invertible elements $a$ and $b$ (see~\cite[Theorem 4.3]{XCZ}).
In ~\cite[Theorem 4.1]{ZCX}, Zhou et al. considered the core inverse of $a+b$ under the condition $a^2a^{\tiny\textcircled{\#}}b^{\tiny\textcircled{\#}}b=baa^{\tiny\textcircled{\#}}, ab^{\tiny\textcircled{\#}}b=aa^{\tiny\textcircled{\#}}b$ in a Dedekind-finite ring in which $2$ is invertible. These inspire us to investigate new additive properties for core invertibility in a Banach *-algebra.

Recall that $a\in \mathcal{A}^{EP}$, i.e., $a$ is an EP element, if there exists $x\in \mathcal{A}$ such that
$xa^2=a, ax=xa, (ax)^*=ax.$ Evidently, $a\in R^{EP}$ if and only if $a\in \mathcal{A}^{\#}$ and $aR=a^*R$(see~\cite[Theorem 3.1]{MDK}).
In Section 2, we present necessary and sufficient conditions under which
the sum of an EP element and core invertible element in a Banach *-algebra is core invertible. Explicit conditions are obtained for the core invertibility of the sum of two EP elements. This extend ~\cite[Theorem 4.3]{XCZ} to the general setting.

In Section 3, we give more simpler conditions for the existence of core inverse under certain communicative conditions. Let $a,b\in \mathcal{A}^{\tiny\textcircled{\#}}$. If $ab=ba$ and $a^*b=ba^*$, we prove that $a+b\in \mathcal{A}^{\tiny\textcircled{\#}}$ and $a^{\tiny\textcircled{\#}}(a+b)^{\tiny\textcircled{\#}}a^{\pi}=0$ if and only if $1+a^{\tiny\textcircled{\#}}b\in \mathcal{A}^{\tiny\textcircled{\#}}$.

Let $C^{n\times n}$ be a Banach *-algebra of $n\times n$ complex matrices, with conjugate transpose as the involution. A matrix $A\in C^{n\times n}$ has core inverse $X$ if and only if $AX=P_A$ and $\mathcal{R}(X)\subseteq \mathcal{R}(A)$, where $P_A$ is the projection on $\mathcal{R}(A)$ (see~\cite[Definition 1]{BT}).
Finally, in Section 4, we apply our preceding results to complex matrices and obtain certain conditions under which a block complex matrix has core inverse.

If $a\in \mathcal{A}$ has the group inverse $a^{\#}$, the element $a^{\pi}=1-aa^{\#}$ is called the spectral idempotent of $a$. Let $p\in \mathcal{A}$ be an idempotent, and let $x\in \mathcal{A}$. Then we write $x=pxp+px(1-p)+(1-p)xp+(1-p)x(1-p),$ and induce a Pierce representation given by the matrix
$x=\left(\begin{array}{cc}
pxp&px(1-p)\\
(1-p)xp&(1-p)x(1-p)
\end{array}
\right)_p.$

Throughout the paper, all Banach *-algebras are complex with an identity. $p\in \mathcal{A}$ is a projection provided that $p^2=p=p^*$.
We use $\mathcal{A}^{\#}, \mathcal{A}^{\tiny\textcircled{\#}}$ and $\mathcal{A}^{EP}$ to
denote the set of all group invertible elements, core invertible elements  and EP elements in $\mathcal{A}$, respectively. $A^*$ stands for the conjugate transpose
$\overline{A}^T$ of the complex matrix $A$.

\section{additive properties}

This section is devoted to study the core inverse of the sum of two core invertible elements in a Banach *-algebra. We begin with

\begin{lem} Let $p\in \mathcal{A}$ be an idempotent, $a\in \mathcal{A}^{\#}$ and $pap^{\pi}=0$. If $ap^{\pi}\in \mathcal{A}^{\#}$, then $(ap^{\pi})(ap^{\pi})^{\#}=(aa^{\#})p^{\pi}.$\end{lem}
\begin{proof} Since $pap^{\pi}=0$, we see that $$\begin{array}{lll}
ap^{\pi}-(a^{\#}p^{\pi})(ap^{\pi})^2&=&ap^{\pi}-(a^{\#}a^2)p^{\pi}\\
&=&(a-a^{\#}a^2)p^{\pi}\\
&=&0.
\end{array}$$ Hence, $ap^{\pi}=(a^{\#}p^{\pi})(ap^{\pi})^2$. It follows that
$$\begin{array}{lll}
(ap^{\pi})(ap^{\pi})^{\#}&=&(a^{\#}p^{\pi})(ap^{\pi})^2(ap^{\pi})^{\#}\\
&=&(a^{\#}p^{\pi})(ap^{\pi}))\\
&=&(aa^{\#})p^{\pi},
\end{array}$$ as required.\end{proof}

\begin{lem} Let $a\in \mathcal{A}^{\tiny\textcircled{\#}}$ and $b\in \mathcal{A}$. Then the following are equivalent:\end{lem}
\begin{enumerate}
\item [(1)] $(1-a^{\tiny\textcircled{\#}}a)b=0$.
\vspace{-.5mm}
\item [(2)] $(1-aa^{\tiny\textcircled{\#}})b=0$.
\end{enumerate}
\begin{proof} See~\cite[Lemma 2.4]{XS}.\end{proof}

\begin{lem} Let $p\in \mathcal{A}$ be a projection, $a\in \mathcal{A}$ and $pap^{\pi}=0$. If $pap, ap^{\pi}\in \mathcal{A}^{\tiny\textcircled{\#}}$ and $(ap^{\pi})^{\pi}p^{\pi}ap=0,$ then $a\in \mathcal{A}^{\tiny\textcircled{\#}}$ and $pa^{\tiny\textcircled{\#}}p^{\pi}=0$.\end{lem}
\begin{proof}  Since $pap^{\pi}=0$, we have $a=\left(\begin{array}{cc}
pap&0\\
p^{\pi}ap&ap^{\pi}
\end{array}
\right)_p.$ By virtue of Theorem 1.1, $pap, ap^{\pi}\in \mathcal{A}^{\#}$. As $(ap^{\pi})^{\pi}p^{\pi}apap=0,$ it follows by ~\cite[Proposition 2.3]{XS} that
$a\in \mathcal{A}^{\#}$. Moreover, we have
$$a^{\#}=\left(
\begin{array}{cc}
(pap)^{\#}&0\\
z&(ap^{\pi})^{\#}
\end{array}
\right)_p,$$ where $$z=[(ap^{\pi})^{\#}]^2p^{\pi}ap(pap)^{\pi}+(ap^{\pi})^{\pi}p^{\pi}ap[(pap)^{\#}]^2-(ap^{\pi})^{\#}p^{\pi}ap(pap)^{\#}.$$
Since $p\in \mathcal{A}$ is a projection, we have $$(pap)^{\tiny\textcircled{\#}}=[p\cdot (pap)\cdot p]^{\tiny\textcircled{\#}}=p(pap)^{\tiny\textcircled{\#}}p\subseteq p\mathcal{A}p.$$ Similarly, $$(ap^{\pi})^{\tiny\textcircled{\#}}=(p^{\pi}ap^{\pi})^{\tiny\textcircled{\#}}\subseteq p^{\pi}\mathcal{A}p^{\pi}.$$
Let $$x=\left(
\begin{array}{cc}
(pap)^{\tiny\textcircled{\#}}&0\\
-(ap^{\pi})^{\tiny\textcircled{\#}}(p^{\pi}ap)(pap)^{\tiny\textcircled{\#}}&(ap^{\pi})^{\tiny\textcircled{\#}}
\end{array}
\right)_p.$$ Then we have
$$\begin{array}{rll}
ax&=&\left(\begin{array}{cc}
pap&0\\
p^{\pi}ap&ap^{\pi}
\end{array}
\right)_p\left(
\begin{array}{cc}
(pap)^{\tiny\textcircled{\#}}&0\\
-(ap^{\pi})^{\tiny\textcircled{\#}}(p^{\pi}ap)(pap)^{\tiny\textcircled{\#}}&(ap^{\pi})^{\tiny\textcircled{\#}}
\end{array}
\right)_p\\
&=&\left(
\begin{array}{cc}
(pap)(pap)^{\tiny\textcircled{\#}}&0\\
0&(ap^{\pi})(ap^{\pi})^{\tiny\textcircled{\#}}
\end{array}
\right)_p.
\end{array}$$ Hence $(ax)^*=[(pap)(pap)^{\tiny\textcircled{\#}}+(ap^{\pi})(ap^{\pi})^{\tiny\textcircled{\#}}]^*=ax$.
We further verify that $$\begin{array}{rll}
axa&=&\left(
\begin{array}{cc}
(pap)(pap)^{\tiny\textcircled{\#}}&0\\
0&(ap^{\pi})(ap^{\pi})^{\tiny\textcircled{\#}}
\end{array}
\right)_p\left(\begin{array}{cc}
pap&0\\
p^{\pi}ap&ap^{\pi}
\end{array}
\right)_p\\
&=&a,
\end{array}$$ and so $a\in \mathcal{A}^{(1,3)}$. According to ~\cite[Lemma 2.1]{XS}, $a$ has core inverse.

Moreover, we have $$\begin{array}{rl}
&a^{\tiny\textcircled{\#}}=a^{\#}ax=\left(
\begin{array}{cc}
(pap)^{\#}&0\\
z&(ap^{\pi})^{\#}
\end{array}
\right)_p\left(\begin{array}{cc}
pap&0\\
p^{\pi}ap&ap^{\pi}
\end{array}
\right)_p\\
&\left(
\begin{array}{cc}
(pap)^{\tiny\textcircled{\#}}&0\\
-(ap^{\pi})^{\tiny\textcircled{\#}}(p^{\pi}ap)(pap)^{\tiny\textcircled{\#}}&(ap^{\pi})^{\tiny\textcircled{\#}}
\end{array}
\right)_p\\
=&\left(
\begin{array}{cc}
*&0\\
*&*
\end{array}
\right)_p.
\end{array}$$ Therefore $pa^{\tiny\textcircled{\#}}p^{\pi}=0$, as asserted.\end{proof}

We are ready to prove:

\begin{thm} Let $a\in \mathcal{A}^{\tiny {EP}}, b, ba^{\pi}\in \mathcal{A}^{\tiny\textcircled{\#}}$ and $aba^{\pi}=0$. Then the following are equivalent:\end{thm}
\begin{enumerate}
\item [(1)] $a+b\in \mathcal{A}^{\tiny\textcircled{\#}}$ and $a(a+b)^{\tiny\textcircled{\#}}a^{\pi}=0$.
\vspace{-.5mm}
\item [(2)] $a(1+a^{\#}b)\in \mathcal{A}^{\tiny\textcircled{\#}}$ and $b^{\pi}a^{\pi}b=0$.
\end{enumerate}
\begin{proof} Let $p=aa^{\#}$. Then $$a=\left(
\begin{array}{cc}
a_1&0\\
0&0
\end{array}
\right)_p, b=\left(
\begin{array}{cc}
b_1&0\\
b_3&b_4
\end{array}
\right)_p.$$
Then $$a+b=\left(
\begin{array}{cc}
a_1+b_1&0\\
b_3&b_4
\end{array}
\right)_p.$$ Here, $$a_1+b_1=aa^{\#}(a+b), b_4=a^{\pi}(a+b)a^{\pi}=ba^{\pi}.$$
We compute that $b_4^{\tiny\textcircled{\#}}b_4=b_4^{\#}b_4b_4^{(1,3)}b_4=b_4^{\#}b_4.$

$(1)\Rightarrow (2)$ Let $x=a+b$. We write
$$x=\left(
\begin{array}{cc}
c_1&0\\
c_2&c_3
\end{array}
\right)_p, x^{\tiny\textcircled{\#}}=\left(
\begin{array}{cc}
x_1&0\\
x_2&x_3
\end{array}
\right)_p.$$ In light of Theorem 1.1, we have
$$x=xx^{\tiny\textcircled{\#}}x, (xx^{\tiny\textcircled{\#}})^*=xx^{\tiny\textcircled{\#}}, (x^{\tiny\textcircled{\#}})^2x=x, x(x^{\tiny\textcircled{\#}})^2=x^{\tiny\textcircled{\#}}.$$
Hence, $c_1=c_1x_1c_1, (c_1x_1)^*=c_1x_1, x_1c_1^2=c_1, c_1x_1^2=x_1.$ Therefore $c_1=aa^{\#}(a+b)aa^{\#}=a(1+a^{\#}b)\in R^{\tiny\textcircled{\#}}$.

In view of \cite[Lemma 2.1]{XS}, $aa^{\#}(a+b)=a_1+b_1$ has group inverse. By virtue of ~\cite[Proposition 2.3]{XS}, we have
$$\begin{array}{lll}
(a+b)^{\#}&=&\left(
\begin{array}{cc}
w^{\#}&0\\
u&b^{\#}a^{\pi}
\end{array}
\right)_p,
\end{array}$$ where $w=aa^{\#}(a+b)$ and $$\begin{array}{lll}
u&=&(b_4^{\#})^2b_3w^{\pi}+b_4^{\pi}b_3(w^{\#})^2-b_4^{\#}b_3w^{\#}\\
&=&b^{\#}a^{\pi}b^{\#}a^{\pi}baa^{\#}w^{\pi}+b^{\pi}a^{\pi}b(w^{\#})^2-b^{\#}a^{\pi}bw^{\#}.
\end{array}$$
We check that
$$\begin{array}{ll}
&(a+b)^{\pi}(a+b)\\
=&\left(
\begin{array}{cc}
w^{\pi}&0\\
-bw^{\#}-ba^{\pi}u&b^{\pi}a^{\pi}
\end{array}
\right)_p\left(
\begin{array}{cc}
w&0\\
a^{\pi}baa^{\#}&ba^{\pi}
\end{array}
\right)_p\\
=&\left(
\begin{array}{cc}
w^{\pi}w&0\\
-(bw^{\#}+ba^{\pi}u)w+b^{\pi}a^{\pi}baa^{\#}&b^{\pi}a^{\pi}ba^{\pi}\\
\end{array}
\right)_p\\
=&0.
\end{array}$$ Hence, $-(bw^{\#}+ba^{\pi}u)w+b^{\pi}a^{\pi}baa^{\#}=0$. Then $b^{\pi}a^{\pi}baa^{\#}=b^{\pi}[-(bw^{\#}+ba^{\pi}u)w+b^{\pi}a^{\pi}baa^{\#}]=0.$
On the other hand, $b^{\pi}a^{\pi}ba^{\pi}=b^{\pi}ba^{\pi}-b^{\pi}a^{\#}aba^{\pi}=0$. Therefore $b^{\pi}a^{\pi}b=b^{\pi}a^{\pi}b[aa^{\#}+a^{\pi}]=0$.

$(2)\Rightarrow (1)$ We verify that
$$\begin{array}{lll}
p^{\pi}[1-b_4^{\tiny\textcircled{\#}}b_4]b_3&=&a^{\pi}[1-b_4^{\#}b_4]a^{\pi}baa^{\#}\\
&=&a^{\pi}[1-(ba^{\pi})(ba^{\pi})^{\#}]a^{\pi}baa^{\#}\\
&=&a^{\pi}[1-bb^{\#}a^{\pi}]a^{\pi}baa^{\#}\\
&=&a^{\pi}b^{\pi}a^{\pi}baa^{\#}\\
&=&0.
\end{array}$$
Hence,
$$\begin{array}{rll}
p(a+b)p^{\pi}&=&a^{\#}(aba^{\pi})=0,\\
p(a+b)p&=&a+aa^{\#}baa^{\#}=a(1+a^{\#}b)\in R^{\tiny\textcircled{\#}},\\
(a+b)p^{\pi}&=&ba^{\pi}\in R^{\tiny\textcircled{\#}}.
\end{array}$$
By virtue of ~\cite[Lemma 2.1]{XS}, $a+b\in \mathcal{A}^{\#}$. This implies that
$$\begin{array}{rll}
((a+b)p^{\pi})^{\pi}p^{\pi}(a+b)p&=&[1-(a+b)p^{\pi}((a+b)p^{\pi})^{\#}]p^{\pi}(a+b)p\\
&=&[1-p^{\pi}(a+b)(a+b)^{\#}]p^{\pi}(a+b)p\\
&=&p^{\pi}[1-(a+b)(a+b)^{\#}]p^{\pi}(a+b)p\\
&=&p^{\pi}[1-b_4^{\tiny\textcircled{\#}}b_4]b_3\\
&=&0.
\end{array}$$ According to Lemma 2.3, $a+b\in R^{\tiny\textcircled{\#}}$ and $a(a+b)^{\tiny\textcircled{\#}}a^{\pi}=a^{\#}p(a+b)^{\tiny\textcircled{\#}}p^{\pi}=0$.\end{proof}

\begin{cor} Let $a\in \mathcal{A}^{\tiny {EP}},b\in \mathcal{A}^{\tiny\textcircled{\#}}$. If $ab=ba$ and $a^*b=ba^*$, then $a+b\in \mathcal{A}^{\tiny\textcircled{\#}}$ if and only if $a(1+a^{\#}b)\in \mathcal{A}^{\tiny\textcircled{\#}}$.\end{cor}
\begin{proof} Since $ab=ba$ and $a^*b=ba^*$, we have $a^{\pi}(a+b)=(a+b)a^{\pi}$. Since $a\in \mathcal{A}^{\tiny {EP}}$, we have $(a^{\pi})^*=a^{\pi}$ by ~\cite[Proposition 2.5]{KP}. Hence $a^{\pi}(a+b)^*=(a+b)^*a^{\pi}$. In light of Lemma 2.1 that $a^{\pi}(a+b)^{\tiny\textcircled{\#}}=(a+b)^{\tiny\textcircled{\#}}a^{\pi}$.
Accordingly, $a(a+b)^{\tiny\textcircled{\#}}a^{\pi}=0$. Clearly, $(ba^{\pi})^{\tiny\textcircled{\#}}=b^{\tiny\textcircled{\#}}a^{\pi}$.
This completes the proof by Theorem 2.4.\end{proof}

We come now to the main result of this section.

\begin{thm} Let $a,b\in \mathcal{A}^{EP},ab^{\pi},ba^{\pi}\in \mathcal{A}^{\tiny\textcircled{\#}}$, $aba^{\pi}=bab^{\pi}=0$. Then the following are equivalent:\end{thm}
\begin{enumerate}
\item [(1)] $a+b\in \mathcal{A}^{\tiny\textcircled{\#}}, a(a+b)^{\tiny\textcircled{\#}}a^{\pi}=ba^{\pi}(a+b)^{\tiny\textcircled{\#}}b^{\pi}=b(a+b)^{\tiny\textcircled{\#}}b^{\pi}=0$.
\vspace{-.5mm}
\item [(2)] $aa^{\#}b+bb^{\#}a\in \mathcal{A}^{\tiny\textcircled{\#}}$ and $a^{\pi}b^{\pi}a=b^{\pi}a^{\pi}b=0$.
\end{enumerate}
\begin{proof}  $(1)\Rightarrow (2)$ In view of Theorem 2.4, $a(1+a^{\#}b)\in \mathcal{A}^{\tiny\textcircled{\#}}$ and $b^{\pi}a^{\pi}b=0$. Analogously,
$a^{\pi}b^{\pi}a=0$. As in the proof in Theorem 2.4, we see that
$(a(1+a^{\#}b))^{\tiny\textcircled{\#}}=aa^{\#}(a+b)^{\tiny\textcircled{\#}}aa^{\#}$.
Let $q=aa^{\#}bb^{\#}$. Since $aba^{\pi}=0$, we have $b=\left(
\begin{array}{cc}
aa^{\#}b&0\\
a^{\pi}baa^{\#}&ba^{\pi}
\end{array}
\right)_{aa^{\#}}$. In view of ~\cite[Lemma 2.1]{XS}, $ba^{\pi}\in \mathcal{A}^{\#}$. Then
$aa^{\#}b\in \mathcal{A}^{\#}$ and $(aa^{\#}b)^{\#}=aa^{\#}b^{\#}$ by ~\cite[Theorem 2.3]{MD}. Hence $q^2=q=(aa^{\#}b)(aa^{\#}b)^{\#}$.
Thus $$\begin{array}{rll}
q(a(1+a^{\#}b))(1-q)&=&aa^{\#}bb^{\#}ab^{\pi}+aa^{\#}bb^{\#}aa^{\#}b(1-aa^{\#}bb^{\#})\\
&=&aa^{\#}bb^{\#}aa^{\#}b-aa^{\#}bb^{\#}aa^{\#}b[1-a^{\pi}]bb^{\#})\\
&=&0,\\
q(a(1+a^{\#}b))^{\tiny\textcircled{\#}}(1-q)&=&qaa^{\#}(a+b)^{\tiny\textcircled{\#}}aa^{\#}(1-q)\\
&=&qaa^{\#}(a+b)^{\tiny\textcircled{\#}}aa^{\#}(1-aa^{\#}bb^{\#})\\
&=&qaa^{\#}(a+b)^{\tiny\textcircled{\#}}aa^{\#}b^{\pi}\\
&=&qaa^{\#}(a+b)^{\tiny\textcircled{\#}}b^{\pi}-qaa^{\#}(a+b)^{\tiny\textcircled{\#}}a^{\pi}b^{\pi}\\
&=&aa^{\#}bb^{\#}aa^{\#}(a+b)^{\tiny\textcircled{\#}}b^{\pi}\\
&=&0.
\end{array}$$
We may write
$$a+aa^{\#}b=\left(
\begin{array}{cc}
c_1&0\\
c_2&c_3
\end{array}
\right)_q, [a+aa^{\#}b]^{^{\tiny\textcircled{\#}}}=\left(
\begin{array}{cc}
x_1&0\\
x_2&x_3
\end{array}
\right)_q.$$
As in the proof of Theorem 2.4, we prove that $c_1^{\tiny\textcircled{\#}}=x_1$.
Moreover, we have
$$\begin{array}{rll}
c_1&=&aa^{\#}bb^{\#}a[1+a^{\#}b]aa^{\#}bb^{\#}\\
&=&aa^{\#}bb^{\#}abb^{\#}+aa^{\#}bb^{\#}aa^{\#}baa^{\#}bb^{\#}\\
&=&aa^{\#}bb^{\#}a+aa^{\#}bb^{\#}aa^{\#}b\\
&=&aa^{\#}(1-b^{\pi})a+aa^{\#}bb^{\#}(1-a^{\pi})b\\
&=&aa^{\#}a-aa^{\#}b^{\pi}a+aa^{\#}bb^{\#}b-aa^{\#}bb^{\#}a^{\pi}b\\
&=&a-b^{\pi}a+aa^{\#}b\\
&=&aa^{\#}b+bb^{\#}a.
\end{array}$$ Therefore $aa^{\#}b+bb^{\#}a\in R^{\tiny\textcircled{\#}}$.

$(2)\Rightarrow (1)$ Let $q=aa^{\#}bb^{\#}$. Then $q^2=q\in \mathcal{A}$.
Moreover, $q=aa^{\#}bb^{\#}aa^{\#}bb^{\#}=(1-a^{\pi})bb^{\#}aa^{\#}bb^{\#}=bb^{\#}aa^{\#}bb^{\#}-a^{\pi}bb^{\#}aa^{\#}=bb^{\#}aa^{\#}bb^{\#}$.
Since $a,b\in \mathcal{A}^{EP}$, we have $$(aa^{\#})^*=aa^{\#},(bb^{\#})^*=bb^{\#}.$$ Then $q^*=q$, i.e., $q\in \mathcal{A}$ is a projection.
In view of Theorem 2.4, it will suffice to prove that $a+aa^{\#}b=a(1+a^{\#}a)\in R^{\tiny\textcircled{\#}}$.
We check that
$$\begin{array}{rll}
qa(1-q)&=&aa^{\#}bb^{\#}a[1-aa^{\#}bb^{\#}]\\
&=&aa^{\#}bb^{\#}a-aa^{\#}bb^{\#}a\\
&=&aa^{\#}(1-b^{\pi})a-aa^{\#}bb^{\#}a\\
&=&a-(1-a^{\pi})b^{\pi}a-aa^{\#}bb^{\#}a\\
&=&a-b^{\pi}a-aa^{\#}(1-b^{\pi})a\\
&=&-b^{\pi}a+(1-a^{\pi})b^{\pi}a\\
&=&0,\\
qaa^{\#}b(1-q)&=&aa^{\#}bb^{\#}aa^{\#}b[1-aa^{\#}bb^{\#}]\\
&=&aa^{\#}bb^{\#}[aa^{\#}b-aa^{\#}baa^{\#}bb^{\#}]\\
&=&aa^{\#}bb^{\#}[aa^{\#}b-aa^{\#}b(1-a^{\pi})bb^{\#}]\\
&=&aa^{\#}bb^{\#}[aa^{\#}b-aa^{\#}b^2b^{\#}-aa^{\#}ba^{\pi}bb^{\#}]\\
&=&0,\\
\end{array}$$
$$\begin{array}{rll}
(1-q)aa^{\#}bq&=&[1-aa^{\#}bb^{\#}]aa^{\#}baa^{\#}bb^{\#}\\
&=&[1-aa^{\#}bb^{\#}]aa^{\#}b\\
&=&aa^{\#}b-aa^{\#}bb^{\#}(1-a^{\pi})b\\
&=&aa^{\#}(1-b^{\pi})a^{\pi})b\\
&=&0,\\
(1-q)aa^{\#}b(1-q)&=&[1-aa^{\#}bb^{\#}]aa^{\#}b[1-aa^{\#}bb^{\#}]\\
&=&[aa^{\#}b-aa^{\#}bb^{\#}(1-a^{\pi})b][1-aa^{\#}bb^{\#}]\\
&=&[aa^{\#}bb^{\#}a^{\pi})b][1-aa^{\#}bb^{\#}]\\
&=&[aa^{\#}(1-b^{\pi})a^{\pi})b][1-aa^{\#}bb^{\#}]\\
&=&0.
\end{array}$$
Then $$a=\left(
\begin{array}{cc}
a_1&0\\
a_3&a_4
\end{array}
\right)_q, aa^{\#}b=\left(
\begin{array}{cc}
b_1&0\\
0&0
\end{array}
\right)_q.$$ Here $a_1=qaq, a_3=q^{\pi}aq, a_4=q^{\pi}aq^{\pi}$.
Then $$a+aa^{\#}b=\left(
\begin{array}{cc}
a_1+b_1&0\\
a_3&a_4
\end{array}
\right)_q.$$ Here, we have
$$\begin{array}{rll}
a_4&=&(1-aa^{\#}bb^{\#})a(1-aa^{\#}bb^{\#})\\
&=&(1-aa^{\#}bb^{\#})ab^{\pi}\\
&=&ab^{\pi})-aa^{\#}bb^{\#}ab^{\pi}\\
&=&ab^{\pi}\\
&\in&\mathcal{A}^{\tiny\textcircled{\#}}.
\end{array}$$
Set $z=a+aa^{\#}b$. Then $zq^{\pi}=[a+aa^{\#}b][1-aa^{\#}bb^{\#}]=ab^{\pi}+aa^{\#}b-aa^{\#}baa^{\#}bb^{\#}=
ab^{\pi}+aa^{\#}b-aa^{\#}b(1-a^{\pi})bb^{\#}=ab^{\pi}\in R^{\tiny\textcircled{\#}}$. We verify that
$$\begin{array}{lll}
qzq&=&qaq+qaa^{\#}bq\\
&=&aa^{\#}bb^{\#}a+aa^{\#}bb^{\#}aa^{\#}b\\
&=&aa^{\#}(1-b^{\pi})a+aa^{\#}bb^{\#}(1-a^{\pi})b\\
&=&a-(1-a^{\pi})b^{\pi}a+aa^{\#}b-aa^{\#}(1-b^{\pi})a^{\pi}b\\
&=&a-b^{\pi}a+aa^{\#}b+aa^{\#}b^{\pi}a^{\pi}b\\
&=&aa^{\#}b+bb^{\#}a\\
&\in &R^{\tiny\textcircled{\#}}.
\end{array}$$
Moreover, we have
$$\begin{array}{rll}
qzq^{\pi}&=&qa(1-q)+qaa^{\#}b(1-q)=0,\\
(zq^{\pi})^{\pi}q^{\pi}zq&=&q^{\pi}a_4^{\pi}a_3\\
&=&q^{\pi}(ab^{\pi})^{\pi}q^{\pi}aq\\
&=&q^{\pi}[1-aa^{\#}b^{\pi}][1-aa^{\#}bb^{\#}]abb^{\#}\\
&=&q^{\pi}[1-aa^{\#}b^{\pi}][a-aa^{\#}bb^{\#}a]\\
&=&q^{\pi}[a-[1-a^{\pi}]b^{\pi}a][1-a^{\#}bb^{\#}a]\\
&=&[1-a^{\#}bb^{\#}]bb^{\#}[1-aa^{\#}bb^{\#}]a\\
&=&[1-a^{\#}bb^{\#}][bb^{\#}-bb^{\#}aa^{\#}bb^{\#}]a\\
&=&[1-a^{\#}bb^{\#}][1-b^{\pi}]a^{\pi}bb^{\#}a\\
&=&[1-a^{\#}bb^{\#}]a^{\pi}[1-b^{\pi}]a\\
&=&0.
\end{array}$$ In light of Lemma 2.3, $a+aa^{\#}b=z\in R^{\tiny\textcircled{\#}}$.
Additionally, we have
$$\begin{array}{lll}
(a+aa^{\#}b)^{\tiny\textcircled{\#}}&=&\left(
\begin{array}{cc}
(a_1+b_1)^{\tiny\textcircled{\#}}&0\\
v&a_4^{\tiny\textcircled{\#}}
\end{array}
\right)_q,
\end{array},$$ where $v=-a_4^{\tiny\textcircled{\#}}a_3(a_1+b_1)^{\tiny\textcircled{\#}}.$

In light of Theorem 2.4, we prove that $a+b\in \mathcal{A}^{\tiny\textcircled{\#}}$. By the preceding discussion, we have $$\begin{array}{c}
aa^{\#}bb^{\#}(a+aa^{\#}b)^{\tiny\textcircled{\#}}(1-aa^{\#}bb^{\#})=0,\\
(a+aa^{\#}b)^{\tiny\textcircled{\#}}=aa^{\#}(a+b)^{\tiny\textcircled{\#}}aa^{\#}.
\end{array}$$ Therefore we get $$\begin{array}{rll}
ba^{\pi}(a+b)^{\tiny\textcircled{\#}}b^{\pi}&=&baa^{\#}(a+b)^{\tiny\textcircled{\#}}b^{\pi}\\
&=&baa^{\#}(a+b)^{\tiny\textcircled{\#}}(a^{\pi}+aa^{\#})b^{\pi}\\
&=&b[(bb^{\#}aa^{\#}][aa^{\#}(a+b)^{\tiny\textcircled{\#}}aa^{\#}]aa^{\#}b^{\pi}\\
&=&b[aa^{\#}-b^{\pi}aa^{\#}][aa^{\#}(a+b)^{\tiny\textcircled{\#}}aa^{\#}]aa^{\#}b^{\pi}\\
&=&baa^{\#}bb^{\#}aa^{\#}][aa^{\#}(a+b)^{\tiny\textcircled{\#}}aa^{\#}]aa^{\#}b^{\pi}\\
&=&b[aa^{\#}bb^{\#}](a+b)^{\tiny\textcircled{\#}}[aa^{\#}(1-aa^{\#}bb^{\#})]\\
&=&b[aa^{\#}bb^{\#}](a+b)^{\tiny\textcircled{\#}}[1-aa^{\#}bb^{\#}]\\
&=&0,
\end{array}$$ as asserted.\end{proof}

In \cite{ZCX}, Zhou et al. investigated the core inverse of $a+b$ in a Dedekind finite ring in which $2$ is invertible. We now derive

\begin{cor} Let $a,b\in \mathcal{A}^{EP},ab^{\pi},ba^{\pi}\in \mathcal{A}^{\tiny\textcircled{\#}}$ and $\frac{1}{2}\in \mathcal{A}$. If $aa^{\#}b=bb^{\#}a\in \mathcal{A}^{\tiny\textcircled{\#}}$, then $a+b\in \mathcal{A}^{\tiny\textcircled{\#}}$.\end{cor}
\begin{proof} Since $aa^{\#}b=bb^{\#}a\in \mathcal{A}^{\tiny\textcircled{\#}}$ and $\frac{1}{2}\in \mathcal{A}$, we have
$aa^{\#}b+bb^{\#}a=2aa^{\#}b\in \mathcal{A}^{\tiny\textcircled{\#}}$. Also we have
$aba^{\pi}=aa^{\#}aba^{\pi}=abb^{\#}aa^{\pi}=0$. Similarly, we get $bab^{\pi}=bb^{\#}bab^{\pi}=baa^{\#}bb^{\pi}=0$.
Moreover, $a^{\pi}b^{\pi}a=a^{\pi}bb^{\#}a=a^{\pi}aa^{\#}b=0.$
Similarly, $b^{\pi}a^{\pi}b=0.$ This completes the proof by Theorem 2.6.\end{proof}

\section{communicative conditions}

In this section, the necessary and sufficient conditions under which the sum of two core invertible elements has core inverse are presented under certain commutative conditions. The following lemmas are crucial.

\begin{lem} (see~\cite[Corollary 3.4]{CZP})) Let $a,b\in \mathcal{A}^{\tiny\textcircled{\#}}$. If $ab=ba$ and $a^*b=ba^*$, then $a^{\tiny\textcircled{\#}}b=ba^{\tiny\textcircled{\#}}$. \end{lem}

\begin{lem} (see~\cite[Theorem 4.3]{XCZ})) Let $a,b\in \mathcal{A}^{\tiny\textcircled{\#}}$. If $ab=0$ and $a^*b=0$, then $a+b\in \mathcal{A}^{\tiny\textcircled{\#}}$.\end{lem}

\begin{lem} (see~\cite[Theorem 3.5]{CZP})) Let $a,b\in \mathcal{A}^{\tiny\textcircled{\#}}$. If $ab=ba$ and $a^*b=ba^*$, then $ab\in \mathcal{A}^{\tiny\textcircled{\#}}$ and $(ab)^{\tiny\textcircled{\#}}=a^{\tiny\textcircled{\#}}b^{\tiny\textcircled{\#}}$.\end{lem}

We come now to the demonstration for which this section has been developed.

\begin{thm} Let $a,b\in \mathcal{A}^{\tiny\textcircled{\#}}$. If $ab=ba$ and $a^*b=ba^*$, then the following are equivalent:\end{thm}
\begin{enumerate}
\item [(1)] $a+b\in \mathcal{A}^{\tiny\textcircled{\#}}$ and $a^{\pi}(a+b)^{\tiny\textcircled{\#}}a=0$.
\vspace{-.5mm}
\item [(2)] $1+a^{\tiny\textcircled{\#}}b\in \mathcal{A}^{\tiny\textcircled{\#}}$ and $(1+a^{\tiny\textcircled{\#}}b)^{\pi}a(1-aa^{\tiny\textcircled{\#}})=0$.
\end{enumerate}
\begin{proof} $(1)\Rightarrow (2)$ We observe that $$\begin{array}{rll}
1+a^{\tiny\textcircled{\#}}b&=&[1-aa^{\tiny\textcircled{\#}}]+[aa^{\tiny\textcircled{\#}}+a^{\tiny\textcircled{\#}}b]\\
&=&[1-aa^{\tiny\textcircled{\#}}]+[aa^{\tiny\textcircled{\#}}+ba^{\tiny\textcircled{\#}}]\\
&=&[1-aa^{\tiny\textcircled{\#}}]+ [a+b]a^{\tiny\textcircled{\#}}
\end{array}$$ Since $ab=ba$ and $a^*b=ba^*$, it follows by Lemma 3.1 that $a^{\tiny\textcircled{\#}}b=ba^{\tiny\textcircled{\#}}$. As $a^{\pi}(a+b)^{\tiny\textcircled{\#}}a=0$, by using Lemma 2.2, $p^{\pi}(a+b)^{\tiny\textcircled{\#}}p=p^{\pi}(a+b)^{\tiny\textcircled{\#}}aa^{\#}a^{(1,3)}=0$.
Obviously, $p^{\pi}(a+b)p=0$.
Let $p=aa^{\tiny\textcircled{\#}}$. Then $$a+b=\left(
\begin{array}{cc}
p(a+b)p&p(a+b)p^{\pi}\\
0&p^{\pi}(a+b)p^{\pi}
\end{array}
\right)_p, (a+b)^{\tiny\textcircled{\#}}=\left(
\begin{array}{cc}
\alpha &\beta\\
0&\gamma
\end{array}
\right)_p.$$
As in the proof of Theorem 2.4, $[p(a+b)p]^{\tiny\textcircled{\#}}=\alpha $. That is, $(a+b)aa^{\tiny\textcircled{\#}}\in \mathcal{A}^{\tiny\textcircled{\#}}$.

We easily check that $[(a+b)aa^{\tiny\textcircled{\#}}]a^{\tiny\textcircled{\#}}=a^{\tiny\textcircled{\#}}[(a+b)aa^{\tiny\textcircled{\#}}]$. In view of Lemma 3.1,
$(a+b)a^{\tiny\textcircled{\#}}=[(a+b)aa^{\tiny\textcircled{\#}}]a^{\tiny\textcircled{\#}}\in \mathcal{A}^{\#}$. Set $y=[(a+b)aa^{\tiny\textcircled{\#}}]^{\tiny\textcircled{\#}}.$ Then
$$(a+b)aa^{\tiny\textcircled{\#}}=(a+b)aa^{\tiny\textcircled{\#}}y(a+b)aa^{\tiny\textcircled{\#}}, [(a+b)aa^{\tiny\textcircled{\#}}y]^*=(a+b)aa^{\tiny\textcircled{\#}}y.$$
We verify that
$$\begin{array}{rl}
&[(a+b)a^{\tiny\textcircled{\#}}](a^2a^{\tiny\textcircled{\#}}y)[(a+b)a^{\tiny\textcircled{\#}}]\\
=&[(a+b)aa^{\tiny\textcircled{\#}]}y[(a+b)aa^{\tiny\textcircled{\#}}]a^{\tiny\textcircled{\#}}\\
=&[(a+b)aa^{\tiny\textcircled{\#}]}a^{\tiny\textcircled{\#}}\\
=&(a+b)a^{\tiny\textcircled{\#}},\\
&[(a+b)a^{\tiny\textcircled{\#}}(a^2a^{\tiny\textcircled{\#}}y)]^*\\
=&[(a+b)aa^{\tiny\textcircled{\#}}y]^*\\
=&(a+b)aa^{\tiny\textcircled{\#}}y\\
=&(a+b)a^{\tiny\textcircled{\#}}(a^2a^{\tiny\textcircled{\#}}y).
\end{array}$$ Therefore $(a+b)a^{\tiny\textcircled{\#}}$ has $(1,3)$-inverse $a^2a^{\tiny\textcircled{\#}}y$.
By virtue of~\cite[Theorem 2.6]{XCZ}, $(a+b)a^{\tiny\textcircled{\#}}\in \mathcal{A}^{\tiny\textcircled{\#}}$. Obviously, we have $$[1-aa^{\tiny\textcircled{\#}}][a+b]a^{\tiny\textcircled{\#}}=[1-aa^{\tiny\textcircled{\#}}]^*[a+b]a^{\tiny\textcircled{\#}}=0.$$
According to Lemma 3.2, $1+a^{\tiny\textcircled{\#}}b\in \mathcal{A}^{\tiny\textcircled{\#}}.$

Since $(a+b)(a+b)^{\tiny\textcircled{\#}}(a+b)=a+b$, we have $$p(a+b)p^{\pi}=p(a+b)p\alpha p(a+b)p^{\pi}+[p(a+b)p\beta +p(a+b)p^{\pi}\gamma]p^{\pi}(a+b)p^{\pi}.$$
Since $[(a+b)(a+b)^{\tiny\textcircled{\#}}]^*=(a+b)(a+b)^{\tiny\textcircled{\#}}$, we have $$p(a+b)p\beta +p(a+b)p^{\pi}\gamma.$$ Then
$$p(a+b)p^{\pi}=p(a+b)p\alpha p(a+b)p^{\pi}.$$ In view of Lemma 2.2, we derive
$[p(a+b)p]^{\pi}p(a+b)p^{\pi}=0$, and so
$$(1+a^{\tiny\textcircled{\#}}b)^{\pi}a^2a^{\tiny\textcircled{\#}}a(1-aa^{\tiny\textcircled{\#}}]=0.$$
Therefore $(1+a^{\tiny\textcircled{\#}}b)^{\pi}a(1-aa^{\tiny\textcircled{\#}})=0$.

$(2)\Rightarrow (1)$ Let $z=(1+a^{\tiny\textcircled{\#}}b)^{\tiny\textcircled{\#}}$. Then we verify that
$$\begin{array}{rl}
&[(1+a^{\tiny\textcircled{\#}}b)a][a^{\tiny\textcircled{\#}}z][(1+a^{\tiny\textcircled{\#}}b)a]\\
=&aa^{\tiny\textcircled{\#}}[(1+a^{\tiny\textcircled{\#}}b)z(1+a^{\tiny\textcircled{\#}}b)]a\\
=&aa^{\tiny\textcircled{\#}}[(1+a^{\tiny\textcircled{\#}}b)a\\
=&(1+a^{\tiny\textcircled{\#}}b)a.
\end{array}$$
Since $(1+a^{\tiny\textcircled{\#}}b)aa^{\tiny\textcircled{\#}}=aa^{\tiny\textcircled{\#}}(1+a^{\tiny\textcircled{\#}}b)$ and $(aa^{\tiny\textcircled{\#}})^*=aa^{\tiny\textcircled{\#}}$, we have $$aa^{\tiny\textcircled{\#}}(1+a^{\tiny\textcircled{\#}}b)^*=(1+a^{\tiny\textcircled{\#}}b)^*aa^{\tiny\textcircled{\#}}.$$
In light of Lemma 3.1, we get $aa^{\tiny\textcircled{\#}}z=zaa^{\tiny\textcircled{\#}}.$

Step 1. It is easy to verify that $$\begin{array}{rll}
(a^2a^{\tiny\textcircled{\#}})a^{\tiny\textcircled{\#}}&=&aa^{\tiny\textcircled{\#}}=a^{\tiny\textcircled{\#}}(a^2a^{\tiny\textcircled{\#}}),\\
a^{\tiny\textcircled{\#}}(a^2a^{\tiny\textcircled{\#}})a^{\tiny\textcircled{\#}}&=&a^{\tiny\textcircled{\#}}(aa^{\tiny\textcircled{\#}})=a^{\tiny\textcircled{\#}},\\
(a^2a^{\tiny\textcircled{\#}})a^{\tiny\textcircled{\#}}(a^2a^{\tiny\textcircled{\#}})&=&(aa^{\tiny\textcircled{\#}})(a^2a^{\tiny\textcircled{\#}})
=a^2a^{\tiny\textcircled{\#}}
\end{array}$$
Thus $a^2a^{\tiny\textcircled{\#}}\in \mathcal{A}^{\#}$. In view of Theorem 1.1, $1+a^{\tiny\textcircled{\#}}b\in \mathcal{A}^{\#}$.
Since $(1+a^{\tiny\textcircled{\#}}b)a^2a^{\tiny\textcircled{\#}}=(a+b)aa^{\tiny\textcircled{\#}}=aa^{\tiny\textcircled{\#}}(a+b)=a^2a^{\tiny\textcircled{\#}}(1+a^{\tiny\textcircled{\#}}b)$, it follows by ~\cite[Lemma 2.1]{XS} that
$(1+a^{\tiny\textcircled{\#}}b)a^2a^{\tiny\textcircled{\#}}\in \mathcal{A}^{\#}$ and $$
\begin{array}{rl}
&[(a+b)aa^{\tiny\textcircled{\#}}]^{\pi}\\
=&[(1+a^{\tiny\textcircled{\#}}b)a^2a^{\tiny\textcircled{\#}}]^{\pi}\\
=&1-(1+a^{\tiny\textcircled{\#}}b)a^2a^{\tiny\textcircled{\#}}(1+a^{\tiny\textcircled{\#}}b)^{\#}a^{\tiny\textcircled{\#}}\\
=&1-(1+a^{\tiny\textcircled{\#}}b)(1+a^{\tiny\textcircled{\#}}b)^{\#}aa^{\tiny\textcircled{\#}}.
\end{array}$$

Step 2. We check that
$$\begin{array}{rl}
&[(1+a^{\tiny\textcircled{\#}}b)a^2a^{\tiny\textcircled{\#}}][a^{\tiny\textcircled{\#}}x]\\
=&[(1+a^{\tiny\textcircled{\#}}b)x][aa^{\tiny\textcircled{\#}}]
\end{array}$$ Hence,
$$\begin{array}{rl}
&[(1+a^{\tiny\textcircled{\#}}b)a^2a^{\tiny\textcircled{\#}}(a^{\tiny\textcircled{\#}}x)]^*\\
=&[aa^{\tiny\textcircled{\#}}]^*[(1+a^{\tiny\textcircled{\#}}b)x]^*\\
=&[aa^{\tiny\textcircled{\#}}][(1+a^{\tiny\textcircled{\#}}b)x]\\
=&[(1+a^{\tiny\textcircled{\#}}b)x][aa^{\tiny\textcircled{\#}}]\\
=&[(1+a^{\tiny\textcircled{\#}}b)a^2a^{\tiny\textcircled{\#}}][a^{\tiny\textcircled{\#}}x]
\end{array}$$
So $(1+a^{\tiny\textcircled{\#}}b)a^2a^{\tiny\textcircled{\#}}$ has a $(1,3)$ inverse $a^{\tiny\textcircled{\#}}x$.

Accordingly, $(a+b)aa^{\tiny\textcircled{\#}}=(1+a^{\tiny\textcircled{\#}}b)a^2a^{\tiny\textcircled{\#}}\in \mathcal{A}^{\tiny\textcircled{\#}}.$
Let $p=aa^{\tiny\textcircled{\#}}$. Then $p^{\pi}bp=(1-aa^{\tiny\textcircled{\#}})baa^{\tiny\textcircled{\#}}=(1-aa^{\tiny\textcircled{\#}})aba^{\tiny\textcircled{\#}}=0$. Similarly,
$pbp^{\pi}=0$. So we get $$a=\left(
\begin{array}{cc}
a_1&a_2\\
0&0
\end{array}
\right)_p, b=\left(
\begin{array}{cc}
b_1&0\\
0&b_4
\end{array}
\right)_p.$$
Hence $$a+b=\left(
\begin{array}{cc}
a_1+b_1&a_2\\
0&b_4
\end{array}
\right)_p.$$ Here $a_1+b_1=(a+b)aa^{\tiny\textcircled{\#}}, b_4=p^{\pi}(a+b)p^{\pi}=bp^{\pi}.$
Since $bp^{\pi}=p^{\pi}b, b^*p^{\pi}=(p^{\pi}b)^*=(bp^{\pi})^*=p^{\pi}b^*$. In light of Lemma 3.3, $b_4=bp^{\pi}\in \mathcal{A}^{\tiny\textcircled{\#}}$ and $b_4^{\tiny\textcircled{\#}}=b^{\tiny\textcircled{\#}}p^{\pi}$.

Let $$x=\left(
\begin{array}{cc}
(a_1+b_1)^{\tiny\textcircled{\#}}&-(a_1+b_1)^{\tiny\textcircled{\#}}a_2b_4^{\tiny\textcircled{\#}}\\
0&(b_4)^{\tiny\textcircled{\#}}
\end{array}
\right)_p.$$ Since $(1+a^{\tiny\textcircled{\#}}b)^{\pi}a(1-aa^{\tiny\textcircled{\#}})=0$, we verify that
$$\begin{array}{rl}
&a_2-(a_1+b_1)(a_1+b_1)^{\tiny\textcircled{\#}}a_2\\
=&[1-(a_1+b_1)(a_1+b_1)^{\tiny\textcircled{\#}}]aa^{\tiny\textcircled{\#}}a(1-aa^{\tiny\textcircled{\#}})\\
=&[1-(1+a^{\tiny\textcircled{\#}}b)(1+a^{\tiny\textcircled{\#}}b)^{\#}aa^{\tiny\textcircled{\#}}]aa^{\tiny\textcircled{\#}}a(1-aa^{\tiny\textcircled{\#}})\\
=&(1+a^{\tiny\textcircled{\#}}b)^{\pi}a(1-aa^{\tiny\textcircled{\#}})\\
=&0.
\end{array}$$ That is, $(a_1+b_1)^{\pi}a_2=0$. In view of ~\cite[Theorem 2.5]{XS}, $a+b\in \mathcal{A}^{\#}$ and
$$(a+b)^{\#}=\left(
\begin{array}{cc}
(a_1+b_1)^{\#}&*\\
0&(b_4)^{\#}
\end{array}
\right)_p.$$
Then we we have
$$\begin{array}{rl}
&(a+b)x\\
=&\left(
\begin{array}{cc}
a_1+b_1&a_2\\
0&b_4
\end{array}
\right)_p\left(
\begin{array}{cc}
(a_1+b_1)^{\tiny\textcircled{\#}}&-(a_1+b_1)^{\tiny\textcircled{\#}}a_2b_4^{\tiny\textcircled{\#}}\\
0&b_4^{\tiny\textcircled{\#}}
\end{array}
\right)_p\\
=&\left(
\begin{array}{cc}
(a_1+b_1)(a_1+b_1)^{\tiny\textcircled{\#}}&0\\
0&b_4b_4^{\tiny\textcircled{\#}}
\end{array}
\right)_p.
\end{array}$$ Hence $[(a+b)x]^*=(a+b)x$.
We further verify that $$\begin{array}{rl}
&(a+b)x(a+b)\\
=&\left(
\begin{array}{cc}
(a_1+b_1)(a_1+b_1)^{\tiny\textcircled{\#}}&0\\
0&b_4b_4^{\tiny\textcircled{\#}}
\end{array}
\right)_p\left(
\begin{array}{cc}
a_1+b_1&a_2\\
0&b_4
\end{array}
\right)_p\\
=&a+b.
\end{array}$$ Thus $a+b\in \mathcal{A}^{(1,3)}$. According to ~\cite[Theorem 2.6]{XCZ}, $a+b$ has core inverse.

Moreover, we have $$\begin{array}{rl}
&(a+b)^{\tiny\textcircled{\#}}\\
=&(a+b)^{\#}(a+b)x\\
=&\left(
\begin{array}{cc}
(a_1+b_1)^{\#}&*\\
0&(b_4)^{\#}
\end{array}
\right)_p\left(
\begin{array}{cc}
(a_1+b_1)(a_1+b_1)^{\tiny\textcircled{\#}}&0\\
0&b_4b_4^{\tiny\textcircled{\#}}
\end{array}
\right)_p\\
=&\left(
\begin{array}{cc}
*&*\\
0&*
\end{array}
\right)_p.
\end{array}$$ We infer that $p^{\pi}(a+b)^{\tiny\textcircled{\#}}a=p^{\pi}(a+b)^{\tiny\textcircled{\#}}pa=0$.
In light of Lemma 2.2, $a^{\pi}(a+b)^{\tiny\textcircled{\#}}a=0$, as asserted.\end{proof}

\begin{cor} Let $a\in \mathcal{A}^{EP},b\in \mathcal{A}^{\tiny\textcircled{\#}}$. If $ab=ba$ and $a^*b=ba^*$, then the following are equivalent:\end{cor}
\begin{enumerate}
\item [(1)] $a+b\in \mathcal{A}^{\tiny\textcircled{\#}}$.
\vspace{-.5mm}
\item [(2)] $1+a^{\#}b\in \mathcal{A}^{\tiny\textcircled{\#}}$.
\end{enumerate}
\begin{proof} Since $a\in \mathcal{A}^{EP}$, it follows by \cite[Theorem 3.1]{RD} that $a^{\#}=a^{\tiny\textcircled{\#}}$. Then $(1+a^{\tiny\textcircled{\#}}b)^{\pi}a(1-aa^{\tiny\textcircled{\#}})=0$.
Since $a^{\pi}(a+b)=(a+b)a^{\pi}$ and $(a^{\pi})^*=a^{\pi}$, we have $a^{\pi}(a+b)^*=(a+b)^*a^{\pi}$. In view of Lemma 3.1,
$a^{\pi}(a+b)^{\tiny\textcircled{\#}}=(a+b)^{\tiny\textcircled{\#}}a^{\pi}$. Thus $a^{\pi}(a+b)^{\tiny\textcircled{\#}}a=(a+b)^{\tiny\textcircled{\#}}a^{\pi}a=0$.
This completes the proof by Theorem 3.4.\end{proof}

\section{applications}

Let $M=\left(
  \begin{array}{cc}
    A&B\\
    C&D
  \end{array}
\right)\in {\Bbb C}^{2n\times 2n},$ where $A,B,C,D\in {\Bbb C}^{n\times n}.$ The aim of this section is to present the core invertibility of the block complex matrix $M$ by using our preceding results.

\begin{lem} Let $B,C$ have core inverses. If $B(CB)^{\pi}=0$ and $C(BC)^{\pi}=0$, then $\left(
  \begin{array}{cc}
    0 & B \\
    C & 0
  \end{array}
\right)$ has core inverse. In this case,
$$Q^{\tiny\textcircled{\#}}=\left(
  \begin{array}{cc}
  0& (BC)^{\#}BCC^{\tiny\textcircled{\#}} \\
  (CB)^{\#}CBB^{\tiny\textcircled{\#}}&0
  \end{array}
\right).$$\end{lem}
\begin{proof} Let $Q=\left(
  \begin{array}{cc}
    0 & B \\
    C & 0
  \end{array}
\right)$. As a complex matrix, $CB$ has Drazin inverse. Then $(CB)(CB)^D=(CB)^D(CB),(CB)^D=(CB)^D(CB)$ $(CB)^D.$
Since $B(CB)^{\pi}=0$, we have $CB(CB)^{\pi}$ $=0$. Then $CB$ has group inverse.
Likewise, $BC$ has group inverse. One directly checks
that
$Q^{\#}=\left(
  \begin{array}{cc}
  0& B(CB)^{\#} \\
  C(BC)^{\#}&0
  \end{array}
\right).$ Moreover, we verify that
$$\begin{array}{rl}
&Q\left(
  \begin{array}{cc}
  0&I\\
  I&0
  \end{array}
\right)\left(
  \begin{array}{cc}
  B^{\tiny\textcircled{\#}}&0\\
  0&C^{\tiny\textcircled{\#}}
  \end{array}
\right)Q\\
=&\left(
  \begin{array}{cc}
  0&BB^{\tiny\textcircled{\#}}B\\
  CC^{\tiny\textcircled{\#}}C&0
  \end{array}
\right)\\
=&Q;\\
&(Q\left(
  \begin{array}{cc}
  0&I\\
  I&0
  \end{array}
\right)\left(
  \begin{array}{cc}
  B^{\tiny\textcircled{\#}}&0\\
  0&C^{\tiny\textcircled{\#}}
  \end{array}
\right))^*\\
=&\left(
  \begin{array}{cc}
  BB^{\tiny\textcircled{\#}}&0\\
  0&CC^{\tiny\textcircled{\#}}
  \end{array}
\right)^*\\
=&Q\left(
  \begin{array}{cc}
  0&I\\
  I&0
  \end{array}
\right)\left(
  \begin{array}{cc}
  B^{\tiny\textcircled{\#}}&0\\
  0&C^{\tiny\textcircled{\#}}
  \end{array}
\right).
\end{array}$$ This implies that $Q$ has $(1,3)$-inverse. In light of ~\cite[Lemma 2.1]{XS},
$Q$ has core inverse. In this case,
$$\begin{array}{rll}
Q^{\tiny\textcircled{\#}}&=&Q^{\#}QQ^{(1,3)}\\
&=&\left(
  \begin{array}{cc}
  0& B(CB)^{\#} \\
  C(BC)^{\#}&0
  \end{array}
\right)\left(
  \begin{array}{cc}
  BB^{\tiny\textcircled{\#}}&0\\
  0&CC^{\tiny\textcircled{\#}}
  \end{array}
\right)\\
&=&\left(
  \begin{array}{cc}
  0& B(CB)^{\#}CC^{\tiny\textcircled{\#}} \\
  C(BC)^{\#}BB^{\tiny\textcircled{\#}}&0
  \end{array}
\right)\\
&=&\left(
  \begin{array}{cc}
  0& (BC)^{\#}BCC^{\tiny\textcircled{\#}} \\
  (CB)^{\#}CBB^{\tiny\textcircled{\#}}&0
  \end{array}
\right),
\end{array}$$ as asserted.\end{proof}

\begin{thm} Let $A,B,C$ and $D$ have core inverses. If $AB=BD, DC=CA, A^*B=BD^*, D^*C=CA^*, B(CB)^{\pi}=0$ and $C(BC)^{\pi}=0$ and $A^{\tiny\textcircled{\#}}BD^{\tiny\textcircled{\#}}C$ is nilpotent, then $M$ has core inverse.\end{thm}
\begin{proof} Write $M=P+Q$, where $$P=\left(
  \begin{array}{cc}
    A & 0 \\
    0 & D
  \end{array}
\right), Q=\left(
  \begin{array}{cc}
   0 & B \\
   C & 0
  \end{array}
\right).$$ Since $A$ and $D$ have core inverses, so has $P$, and that $$P^{\tiny\textcircled{\#}}=\left(
  \begin{array}{cc}
    A^{\tiny\textcircled{\#}} & 0 \\
    0 & D^{\tiny\textcircled{\#}}
  \end{array}
\right).$$ In view of Lemma 4.1, $Q$ has core inverse. We easily check that
$$PQ=\left(
  \begin{array}{cc}
   0 & 0 \\
    DC & 0
  \end{array}
\right)=\left(
  \begin{array}{cc}
  0 & 0 \\
    CA & 0
  \end{array}
\right)=QP.$$ Likewise, we verify that $P^*Q=QP^*$. Moreover, we check that
$$\begin{array}{rll}
I_{m+n}+P^{\tiny\textcircled{\#}}Q&=&\left(
\begin{array}{cc}
I_m&A^{\tiny\textcircled{\#}}B\\
D^{\tiny\textcircled{\#}}C&I_n
\end{array}
\right).
\end{array}$$ Since $A^{\tiny\textcircled{\#}}BD^{\tiny\textcircled{\#}}C$ is nilpotent, we prove that
$I_{m+n}+P^{\tiny\textcircled{\#}}Q$ is invertible, and so it has core inverse. Additionally, $[I_{m+n}+P^{\tiny\textcircled{\#}}Q]^{\pi}=0$.
According to Theorem 3.4, $M$ has core inverse, as asserted.
\end{proof}

As an immediate consequence we now derive

\begin{cor} Let $A,B,C$ and $D$ have core inverses. If $AB=BD, DC=CA, A^*B=BD^*, D^*C=CA^*, B(CB)^{\pi}=0$ and $C(BC)^{\pi}=0$ and $A^{\tiny\textcircled{\#}}BD^{\tiny\textcircled{\#}}C$ is nilpotent, then $M$ has core inverse.\end{cor}
\begin{proof} Applying Theorem 4.2, we prove that $\left(
\begin{array}{cc}
D&C\\
B&A
\end{array}
\right)$ has group inverse. Clearly,
$$M=\left(
\begin{array}{cc}
0&I\\
I&0
\end{array}
\right)\left(
\begin{array}{cc}
D&C\\
B&A
\end{array}
\right)\left(
\begin{array}{cc}
0&I\\
I&0
\end{array}
\right),$$ and so $M$ has core inverse.\end{proof}

We are now ready to prove the following.

\begin{thm} Let $A,B,C$ and $D$ have core inverses. If $AB=BD, DC=CA, B^*A=DB^*, B(CB)^{\pi}=0$ and $C(BC)^{\pi}=0$ and $A^{\tiny\textcircled{\#}}BD^{\tiny\textcircled{\#}}C$ is nilpotent, then $M$ has core inverse.\end{thm}
\begin{proof} Write $M=P+Q$, where $$P=\left(
  \begin{array}{cc}
    A & 0 \\
    0 & D
  \end{array}
\right), Q=\left(
  \begin{array}{cc}
   0 & B \\
   C & 0
  \end{array}
\right).$$ Then we check that
$$\begin{array}{rll}
Q^*P&=&\left(
  \begin{array}{cc}
   0 & C^* \\
  B^* & 0
  \end{array}
\right)\left(
  \begin{array}{cc}
    A & 0 \\
    0 & D
  \end{array}
\right)\\
&=&\left(
  \begin{array}{cc}
    0 & C^*D \\
    B^*A & 0
  \end{array}
\right)\\
&=&\left(
  \begin{array}{cc}
    0 & AC^* \\
   DB^* & 0
  \end{array}
\right)\\
&=&\left(
  \begin{array}{cc}
    A & 0 \\
    0 & D
  \end{array}
\right)\left(
  \begin{array}{cc}
   0 & C^* \\
  B^* & 0
  \end{array}
\right)\\
&=&PQ^*.
\end{array}$$ Similarly, $QP=PQ$.
Further, we verify that
$$\begin{array}{rll}
I+Q^{\tiny\textcircled{\#}}P&=&I+\left(
  \begin{array}{cc}
  0& (BC)^{\#}BCC^{\tiny\textcircled{\#}} \\
  (CB)^{\#}CBB^{\tiny\textcircled{\#}}&0
  \end{array}
\right)\left(
  \begin{array}{cc}
    A & 0 \\
    0 & D
  \end{array}
\right)\\
&=&\left(
\begin{array}{cc}
I&A^{\tiny\textcircled{\#}}(BC)^{\#}BCC^{\tiny\textcircled{\#}}D\\
(CB)^{\#}CBB^{\tiny\textcircled{\#}}A&I
\end{array}
\right).
\end{array}$$ Since $A^{\tiny\textcircled{\#}}BD^{\tiny\textcircled{\#}}C$ is nilpotent, we prove that
$I+Q^{\tiny\textcircled{\#}}P$ is invertible; hence, it has core inverse.  Additionally, $[I+Q^{\tiny\textcircled{\#}}P]^{\pi}=0$.
In light of Theorem 3.4, $M$ has core inverse, as requred.\end{proof}

\begin{cor} Let $A,B,C$ and $D$ have core inverses. If $AB=BD, DC=CA, A^*B=BD^*, D^*C=CA^*, B(CB)^{\pi}=0$ and $C(BC)^{\pi}=0$ and $A^{\tiny\textcircled{\#}}BD^{\tiny\textcircled{\#}}C$ is nilpotent, then $M$ has core inverse.\end{cor}
\begin{proof} As in the proof of Corollary 4.3, we obtain the result by applying Theorem 4.4 to the block matrix
$\left(
\begin{array}{cc}
D&C\\
B&A
\end{array}
\right)$.\end{proof}

\end{document}